\documentclass[a4paper]{amsart}

\usepackage[utf8]{inputenc}
\usepackage[british]{babel}
\usepackage{amsfonts,amsmath,amssymb,amsthm,mathtools,esint}
\usepackage{tikz}
\usepackage{enumitem}
\usepackage[hidelinks]{hyperref}
\usepackage{cleveref}
\crefname{subsection}{subsection}{subsections}
\usepackage{thmtools}
\numberwithin{equation}{section}
\usepackage{xcolor}
\definecolor{dark-gray}{gray}{0.3}
\usepackage{caption}

\newcommand{\R}{\mathbb{R}}
\renewcommand{\d}[1]{\,\mathrm{d}#1}

\newcommand{\I}{\mathrm{I}}
\newcommand{\J}{\mathrm{J}}
\newcommand{\G}{\mathrm{G}}

\newcommand{\Fcal}{\mathcal{F}}
\newcommand{\Tcal}{\mathcal{T}}

\newcommand{\Mcal}{\mathcal{M}}

\newcommand{\pw}{\mathrm{pw}}
\newcommand{\D}{\mathrm{D}}
\newcommand{\s}{\mathrm{s}}
\newcommand{\osc}{\mathrm{osc}}

\newtheorem{theorem}{Theorem}[section]
\newtheorem{lemma}[theorem]{Lemma}

\theoremstyle{remark}
\newtheorem{remark}[theorem]{Remark}

\newcounter{cntS}

\newcounter{cntL}

\usepackage[backend=biber, giveninits=true, isbn=false, doi=false, url=false, maxnames = 4]{biblatex}
\DeclareFieldFormat[article]{title}{{#1}} 
\DeclareFieldFormat[book]{title}{{#1}}
\DeclareFieldFormat[inproceedings]{title}{{#1}}
\DeclareFieldFormat[incollection]{title}{{#1}}
\DeclareFieldFormat{pages}{{#1}} 
\AtBeginBibliography{\footnotesize}
\begin{document}

\title[Polytopal FEM for biharmonic equation]{Quasi-optimal polytopal finite element methods for biharmonic equation}

\author[N.~T.~Tran]
{Ngoc Tien Tran}

\thanks{The author received funding from the European Union's Horizon 2020 research and innovation programme (project RandomMultiScales, grant agreement No.~865751).}

\address[N.~T.~Tran]{Universit\"at Augsburg, 86159 Augsburg, Germany}
\email{ngoc1.tran@uni-a.de}
\date{\today}

\keywords{biharmonic, quasi-optimal, discontinuous Galerkin, weak Galerkin, hybrid high-order, minimal regularity}

\subjclass{65N12, 65N15, 65N30}

\begin{abstract}
	This paper establishes quasi-optimal and lower-order error estimates for weak Galerkin, discontinuous Galerkin, and hybrid high-order finite element methods for the biharmonic equation under minimal regularity assumptions on general polytopal meshes. Furthermore, it is shown that the stabilization is an efficient contribution in a~posteriori error estimators.
\end{abstract}

\maketitle
\section{Introduction}
\subsection{Motivation}
Quasi-optimal error estimates arise from the Galerkin orthogonality of conforming finite element methods (FEM) in Cea's lemma \cite{Cea1964}, which is not accessible for nonconforming methods. Instead, we are in the setting of the second Strang lemma \cite{BergerScottStrang1972}. Bounding the consistency error therein often leads to additional regularity assumptions on the exact solution \cite{DiPietroErn2012}.

The medius analysis of \cite{Gudi2010} provides, up to data oscillation, quasi-optimal error estimates for classical nonconforming and discontinuous Galerkin (DG) FEM. If smoothing operators are involved in the definition of the discrete right-hand side, then these nonconforming FEM are quasi-optimal without data oscillation \cite{VeeserZanotti2019II,VeeserZanotti2018III}. Remarkably, quasi-optimality of FEM can be precisely characterized in an abstract setting \cite{VeeserZanotti2018I}. An extension of \cite{VeeserZanotti2018I} to hybrid high-order (HHO) methods for elliptic PDE of second and fourth order is provided in \cite{ErnZanotti2020,CarstensenKhotPani2023,CarstensenTran2025,LiangTran2026}. A crucial ingredient therein is the existence of a continuous right-inverse of the interpolation operator, leading to vanishing discrete consistency error.
This setting is quite restrictive, in particular, it excludes many popular nonconforming methods for the biharmonic problem, where the design of numerical schemes enjoys more flexibility with respect to degrees of freedom.

Beyond a~priori error control, the error analysis of \cite{VeeserZanotti2018III,ErnZanotti2020,CarstensenTran2025,LiangTran2026} also provide the efficiency of the stabilization of the discrete solution in the a~posteriori error analysis. While this stabilization can be avoided on simplicial meshes via Helmholtz decomposition \cite{BertrandCarstensenGraessleT2023,ChaumontFrelet26} or by pushing the stabilization parameter towards infinity \cite{KarakashianPascal2007}, it arises naturally as a contribution in error estimators on general polytopal meshes.

This paper is concerned with the a~priori error analysis of weak Galerkin (WG), DG, and HHO FEM for the biharmonic problem in two or three space dimensions under minimal regularity assumptions.
The popularity of
these methods ensues from their flexibility in discretization degree and mesh design, avoiding the high polynomial degrees of conforming methods.
However, apart from DG methods on simplicial meshes \cite{Gudi2010}, error bounds in the literature still requires higher regularity assumptions that may not be feasible in three space dimensions or in presence of mixed boundary conditions. Furthermore, reliable and efficient a~posteriori error estimators become much simpler to design with the efficiency of the stabilization established in this paper. 

\subsection{Main results}
Given a bounded polyhedral Lipschitz domain $\Omega \subset \mathbb{R}^d$, $d \in \{2,3\}$ and $f \in L^2(\Omega)$, the biharmonic problem (with clamped boundary conditions for the sake of brevity) seeks the solution $u \in V \coloneqq H^2_0(\Omega)$ to
\begin{align}\label{def:cont-problem}
	a(u,v) = F(v) \quad\text{for any } v \in V.
\end{align}
with $a(u,v) \coloneqq (\Delta u, \Delta v)_{L^2(\Omega)} = (\D^2 u, \D^2 v)_{L^2(\Omega)}$ and $F = (f, \bullet)_{L^2(\Omega)} \in V'$.
Discretizations of \eqref{def:cont-problem} seeks the solution $u_h$ in a given discrete ansatz space $V_h$ with
\begin{align}\label{def:dis-problem}
	a_h(u_h,v_h) = F_h(v_h) \quad\text{for any } v_h \in V_h,
\end{align}
where the bilinear form $a_h : V_h \times V_h \to \mathbb{R}$ and $F_h \in V_h'$ are suitable approximations of the continuous objects $a$ and $F$. We assume that $a_h$ is coercive w.r.t.~some discrete norm $\|\bullet\|_h$ of $V_h$, i.e.,
\begin{align}\label{ineq:coercivity}
	\|v_h\|_h^2 \lesssim a_h(v_h,v_h) \quad\text{for any } v_h \in V_h.
\end{align}
Given linear bounded operators $\I_h : V \to V_h$ (interpolation) and $\J_h : V_h \to V$ (smoothing), the error analysis can depart from the split
\begin{align}\label{ineq:err-split}
	\|e_h\|^2_h &\lesssim a_h(\I_h u - u_h, e_h)\nonumber\\
	&= a_h(\I_h u, e_h) - a(u, \J_h e_h) + F(\J_h e_h) - F_h(e_h).
\end{align}
with the abbreviation $e_h \coloneqq \I_h u - u_h \in V_h$.
The error $F(\J_h e_h) - F_h(e_h)$ in data approximation can be bounded using standard techniques and even vanishes if $F_h \coloneqq F \circ \J_h$. Thus, the difficulty in deriving quasi-optimal error estimates lies in the design of $\I_h$ and $\J_h$ so that the consistency error $a_h(\I_h u, e_h) - a(u, \J_h e_h)$ is quasi-optimal.

This paper suggests interpolation operators $\I_h$ based on the Galerkin (also known under the label of elliptic) projection. This is rather non-standard in hybridizable settings, where $L^2$ projections are employed for an interpolation onto cell variables.
While $L^2$ projections appear to be the canonical choice considering the right-inverse property, Galerkin projections enable quasi-optimality of Lehrenfeld-Sch\"oberl typed stabilizations on polytopal meshes.

Furthermore, we construct smoothing operators that are at least almost consistent with discrete test functions in the sense that the resulting error behaves quasi-optimal. These two ingredients are the driving force behind the analysis of this paper and, together with (relatively) standard techniques in medius analysis, we obtain quasi-optimal and lower-order error estimates for the WG method of \cite{MuWangYe2014}, the symmetric and nonsymmetric interior penalty DG methods \cite{MozolevskiSueli2003,GeorgoulisHouston2009}, and the HHO method of \cite{DongErn2022}. The results are of the form
\begin{align}\label{ineq:main_results}
	\|\D^2_\pw(u - \hat{u}_h)\| + \|\I_h u - u_h\|_h + |u_h|_\s &\lesssim \min_{\phi \in P_k(\Mcal)} \|\D^2_\pw(u - \phi)\| + \osc(f,\Mcal),\nonumber\\
	\|u - \hat{u}_h\| &\lesssim h^\delta\big(\min_{\phi \in P_k(\Mcal)} \|\D^2_\pw(u - \phi)\| + \osc(f,\Mcal)\big)
\end{align}
with a discrete approximation $\hat{u}_h$ of $u$ in the space $P_k(\Mcal)$ of piecewise polynomials of degree at most $k$ w.r.t.~some mesh $\Mcal$, data oscillation $\osc(f,\Mcal)$, and positive parameter $\delta$ depending on the elliptic regularity of the problem.
Here and throughout, $\|\bullet\|$ denotes the $L^2$ norm over $\Omega$.
In particular, the stabilization is an efficient contribution in a~posteriori error estimators, i.e.,
$$|u_h|_\s \lesssim \|\D^2_\pw(u - \hat{u}_h)\| + \osc(f,\Mcal).$$ 
If $F_h \coloneqq F \circ \J_h$, then the data oscillations in \eqref{ineq:main_results} can be omitted. Note that this right-hand side is also well-defined for any $F \in V'$.

We emphasize that \eqref{ineq:main_results} holds without additional regularity assumptions and on general polytopal meshes. Furthermore, it is expected that the analysis of this paper extends beyond the presented numerical schemes.

\subsection{Outline}
The remaining parts of this paper are organized as follows.
\Cref{sec:smoothing} constructs conforming approximations of discrete functions in a general setting, which lead, in the following sections, to suitable smoothing operators. Quasi-optimal error estimates for the WG method of \cite{MuWangYe2014} are established in \Cref{sec:WG}, followed by the a~priori error analysis of DG methods in \Cref{sec:DG}.
Corresponding results for the HHO method of \cite{DongErn2022} are provided in \Cref{sec:HHO}.

\subsection{Notation}
Throughout this paper, standard notation for Lebesgue and Sobolev spaces are employed. The $L^2$ norm over $\omega$ is abbreviated by $\|\bullet\|_\omega$ with the convention $\|\bullet\| \coloneqq \|\bullet\|_\Omega$ and $(\bullet, \bullet)_{L^2(\omega)}$ is the $L^2$ scalar product.

\section{Smoothing operator}\label{sec:smoothing}
Before the construction of conforming approximations is presented, we provide details on the polytopal mesh and fix the notation of finite element spaces.

\subsection{Polytopal meshes}
Let $\mathcal{M}$ be a finite collection of closed polytopes of positive volume with overlap of measure zero that covers $\overline{\Omega} = \cup_{K \in \mathcal{M}} K$.
A face $S$ of the mesh $\mathcal{M}$ is a closed connected subset of a hyperplane $H_S$ with positive $(d-1)$-dimensional surface measure such that either (a) there exist $K_+,K_- \in \mathcal{M}$ with $S \subset H_S \cap K_+ \cap K_-$ (interior face) or (b) there exists $K_+ \in \mathcal{M}$ with $S \subset H_S \cap K_+ \cap \partial \Omega$ (boundary face).

Let $\Sigma$ be a finite collection of faces with overlap of $(d-1)$-dimensional surface measure zero that covers the skeleton $\partial \Mcal \coloneqq \cup_{K \in \mathcal{M}} \partial K = \cup_{S \in \Sigma} S$ with the split $\Sigma = \Sigma(\Omega) \cup \Sigma(\partial \Omega)$ into the set of interior faces $\Sigma(\Omega)$ and the set of boundary faces $\Sigma(\partial \Omega)$.
For $K \in \mathcal{M}$, $\Sigma(K)$ is the set of all faces of $K$.
The normal vector $\nu_S$ of an interior face $S \in \Sigma(\Omega)$ is fixed in its orientation beforehand and set $\nu_S \coloneqq \nu|_S$ for boundary faces $S \in \Sigma(\partial \Omega)$.
For $S \in \Sigma(\Omega)$, $K_+ \in \Mcal$ (resp.~$K_- \in \Mcal$) denotes the unique cell with $S \subset \partial K_{+}$ (resp.~$S \subset \partial K_-$) and $\nu_{K_+}|_S = \nu_S$ (resp.~$\nu_{K_-}|_S = -\nu_S$).
For $S \in \Sigma(\partial \Omega)$, $K_+ \in \Mcal$ is the unique cell with $S \subset \partial K_+$.
The jump $[v]_S$ and the average $\{v\}_S$ of any function $v \in W^{1,1}(\mathrm{int}(T_+ \cup T_-))$ along $S \in \Sigma(\Omega)$ are defined by $[v]_S \coloneqq v|_{K_+} - v|_{K_-} \in L^1(S)$ and $\{v\}_S \coloneqq (v|_{K_+} + v|_{K_-})/2 \in L^1(S)$.
If $S \in \Sigma(\partial \Omega)$, then $[v]_S$ and $\{v\}_S$ denote the trace of $v$.

For a piecewise Sobolev function $v \in H^1(\Mcal)$, i.e., $v \in H^1(\mathrm{int}(K))$ for any $K \in \Mcal$, $[v]_\Sigma$ and $\{v\}_\Sigma$ are functions in $L^2(\partial \Mcal)$ defined by $[v]_\Sigma|_S = [v]_S$ and $\{v\}_\Sigma|_S = \{v\}_S$ for any $S \in \Sigma$. 

Differential operators with subscript $\pw$ denote the piecewise version without explicit reference to the underlying mesh.

For theoretical purposes, let $\vartheta$ denote the mesh regularity parameter of $\Mcal$ associated with a matching simplicial submesh $\mathcal{T}$ and the set $\mathcal{F}$ of faces of $\mathcal{T}$, we refer to \cite[Definition 1.38]{DiPietroErn2012} for a detailed definition ($\vartheta$ is the minimum of the two parameters therein). The constants in inverse estimates and trace inequalities depend on $\vartheta$.


\subsection{Finite element spaces}\label{sec:fem-spaces}
Given a subset $M \subset \R^d$ of diameter $h_M$, let $P_k(M)$ denote the space of polynomials of degree at most $k$;
$P_k(\mathcal{M})$ and $P_k(\Sigma)$ are 
the space of piecewise polynomials of degree at most $k$ with respect to the mesh $\mathcal{M}$ and the faces $\Sigma$.
The set $P_{k}(\Sigma(\Omega))$ contains all functions $v_\Sigma \in P_k(\Sigma)$ that vanish on the boundary, namely, $v_\Sigma|_S = 0$ for any $S \in \Sigma(\partial \Omega)$.

For any $v \in L^1(M)$, we denote the $L^2$ projection of $v$ onto $P_k(M)$ by $\Pi_M^k v \in P_k(M)$; $\Pi_\Mcal^k$ and $\Pi_\Sigma^k$ are $L^2$ projections onto $P_k(\Mcal)$ and $P_k(\Sigma)$.

The piecewise constant function $h_\Mcal \in P_0(\Mcal)$ reads $h_\Mcal|_K = h_K = \mathrm{diam}(K)$ and $h_{\max} \coloneqq \max_{K \in \Mcal} h_K$ is the maximal mesh-size of $\Mcal$. Furthermore, we define the piecewise constant function $\nu_\Sigma \in P_0(\Sigma)$ by $\nu_\Sigma|_S = \nu_S$. 

\subsection{Construction of smoothing function}
Let $W_h \coloneqq P_{\ell}(\Mcal) \times P_m(\Sigma(\Omega)) \times P_n(\Sigma(\Omega))$ with $\ell, m, n \geq 0$.

\begin{theorem}[smoothing]\label{thm:smoothing}
	Given $w \in P_k(\Mcal)$ and $(w_\Mcal, w_\Sigma, \delta_\Sigma) \in W_h$, $k \geq 2$, there exists a function $v \in V$ with the following properties.
	\begin{enumerate}[wide]
		\item[(a)] (consistent weights) $\Pi_\Mcal^\ell v = w_\Mcal$, $\Pi_\Sigma^m v = w_\Sigma$, $\Pi_\Sigma^n \nabla v \cdot \nu_\Sigma = \delta_\Sigma$.
		\item[(b)] (local approximation property) For all $K \in \Mcal$,
		\begin{align*}
			&h_K^{-4}\|v - w\|_K^2 + h_K^{-2}\|\nabla(v - w)\|_K^2 + \|\D^2(v - w)\|_K^2\\
			&\quad\lesssim \sum_{S \in \Sigma, S \cap K \neq \emptyset} (h_S^{-3}\|[w]_S\|^2_S + h_S^{-1}\|[\nabla_\pw w]_S \cdot \nu_S\|^2_S) + h_K^{-4}\|\Pi_K^\ell(w - w_\Mcal)\|_K^2\\
			&\quad\qquad + \sum_{S \in \Sigma(K)} \big(h_S^{-3}\|\Pi_S^m(w - w_\Sigma)\|_S^2 + h_S^{-1}\|\Pi_S^n(\nabla_\pw w \cdot \nu_S - \delta_\Sigma)\|_S^2\big).
		\end{align*}
	\end{enumerate}
\end{theorem}

\begin{proof}
	Since the two-dimensional case is simpler, the proof below only focuses on the three dimensional case.
	In the first step, we average $w$ at all degrees of freedom associated with higher-order $C^1$ conforming finite element functions on the Worsey–Farin splits \cite{WorseyFarin1987} of the matching simplicial submesh $\Tcal$. 
	These degrees of freedom only depend on the evaluation of the function $w$ and its first derivative \cite{GeorgoulisHoustonVirtanen2011,GuzmanLischkeNeilan2022}.
	Standard averaging techniques \cite{BrennerGudiSung2010,GeorgoulisHoustonVirtanen2011} lead to $v_0 \in V$ with
	\begin{align*}
		&h_T^{-4}\|w - v_0\|_T^2 + h_T^{-2}\|\nabla(w - v_0)\|_T^2 + \|\D^2(w - v_0)\|_T^2\\
		&\qquad\qquad\lesssim \sum_{F \in \Fcal, F \cap T \neq \emptyset} \big(h_F^{-3} \|[w]_F\|^2_F + h_F^{-1} \|[\nabla_\pw w]_F \cdot \nu_F\|^2_F\big).
	\end{align*}
	For any $K \in \Mcal$, summing over all $T \in \Tcal$ with $T \subset K$ and $h_T \approx h_K$ shows
	\begin{align}\label{ineq:pr-local-smoothing-err-1}
		&h_K^{-4}\|w - v_0\|_K^2 + h_K^{-2}\|\nabla(w - v_0)\|_K^2 + \|\D^2(w - v_0)\|_K^2\nonumber\\
		&\qquad\qquad\lesssim \sum_{S \in \Sigma, S \cap K \neq \emptyset} \big(h_S^{-3} \|[w]_S\|^2_S + h_S^{-1} \|[\nabla_\pw w]_S \cdot \nu_S\|^2_S\big),
	\end{align}
	where we utilize that $w$ and $\nabla_\pw w$ only jump across faces $F \in \Fcal$ on the skeleton $\partial \Mcal$. To establish (a), we add corrector functions in $C^1$ conforming finite element spaces, whose degrees of freedom are known from \cite{Zhang2009} and can be written in integral form \cite[Appendix A.2]{LiangTran2026}. We prescribe the unique function $v_1 \in P_N(\Tcal) \cap V$ with $N = \max\{\ell+8,m+9,n+7\}$ by setting the degrees of freedom (5)--(7)
	in \cite[Appendix A.2]{LiangTran2026} via
	\begin{align*}
		0 &= \int_T (v_1 - \Pi_\Mcal^\ell(v_0 - w_\Mcal)) p \d{x} = \int_F (v_1 - \Pi_\Sigma^m(v_0 - w_\Sigma)|_F) q \d{s}\\
		&= \int_F (\nabla v_1 \cdot \nu_F - \Pi_\Sigma^n(\nabla v_0 \cdot \nu_\Sigma - \delta_\Sigma)|_F) r \d{s} 
	\end{align*}
	for any $T \in \Tcal$, $F \in \Fcal(\Omega)$, $p \in P_{N-8}(T)$, $q \in P_{N-9}(F)$, and $r \in P_{N-7}(F)$,
	while the remaining ones (1)--(4) are set to zero.
	Here, the restriction of an $L^2$ function on the skeleton $\partial \Mcal$ to a face $F$ with $F \not\subset \partial \Mcal$ is defined as zero.
	By equivalence of norms in finite dimensional spaces and inverse estimates,
	\begin{align*}
		&h_T^{-4}\|v_1\|_T^2 + h_T^{-2}\|\nabla v_1\|_T^2 + \|\D^2 v_1\|_T^2 \lesssim h_T^{-4}\|\Pi_\Mcal^\ell(v_0 - w_\Mcal)\|_T^2\\
		&\quad+ \sum_{F \in \Fcal(T)} \big(h_F^{-3}\|\Pi_\Sigma^m(v_0 - w_\Sigma)|_F\|_F^2 + h_F^{-1}\|\Pi_\Sigma^n(\nabla v_0 \cdot \nu_\Sigma - \delta_\Sigma)|_F\|_F^2\big)
	\end{align*}
	holds for any $T \in \Tcal$. The sum of this over $T \in \Tcal$, $T \subset K$ with $h_T \approx h_K$ implies
	\begin{align}\label{ineq:pr-local-smoothing-err-2}
		&h_K^{-4}\|v_1\|_K^2 + h_K^{-2}\|\nabla v_1\|_K^2 + \|\D^2 v_1\|_K^2\lesssim h_K^{-4}\|\Pi_\Mcal^\ell(v_0 - w_\Mcal)\|_K^2\nonumber\\
		&\qquad \lesssim \sum_{S \in \Sigma(K)} \big(h_S^{-3}\|\Pi_\Sigma^m(v_0 - w_\Sigma)\|_S^2 + h_S^{-1}\|\Pi_\Sigma^n(\nabla v_0 \cdot \nu_\Sigma - \delta_\Sigma)\|_S^2\big).
	\end{align}
	The function $v \coloneqq v_0 + v_1$ satisfies (a) by design and (b) from \eqref{ineq:pr-local-smoothing-err-1}--\eqref{ineq:pr-local-smoothing-err-2} and the triangle inequality.
\end{proof}
\begin{remark}[linearity]
	The mapping $(w, w_\Mcal, w_\Sigma, \delta_\Sigma) \mapsto v$ is linear by construction of $v$. Furthermore, $v$ is a piecewise polynomial w.r.t.~the Worsey-Farin splits of $\Tcal$ and, thus, satisfies inverse estimates.
\end{remark}

\section{Weak Galerkin}\label{sec:WG}
This section derives quasi-optimal error estimates for the WG FEM of \cite{MuWangYe2014}.

\subsection{Discrete problem}
We consider the ansatz space $V_h \coloneqq P_k(\Mcal) \times P_k(\Sigma(\Omega)) \times P_{k-1}(\Sigma(\Omega))$ with $k \geq 2$.
Given $v_h = (v_\Mcal, v_\Sigma, \gamma_\Sigma) \in V_h$, the discrete Laplacian $\Delta_h v_h \in P_{k-2}(\Mcal)$ of $v_h$ is the unique solution to
\begin{align}\label{def:discrete-Laplace-WG}
	&(\Delta_h v_h, \phi)_{L^2(\Omega)}\nonumber\\
	&\qquad = (v_\Mcal, \Delta_\pw \phi)_{L^2(\Omega)} + \sum_{S \in \Sigma(\Omega)} \int_S \big(\gamma_\Sigma [\phi]_S - v_\Sigma [\nabla_\pw \phi]_S \cdot \nu_S \big) \d{s}
\end{align}
for any $\phi \in P_{k-2}(\Mcal)$. The discrete problem seeks the discrete solution $u_h = (u_\Mcal, u_\Sigma, \beta_\Sigma) \in V_h$ to \eqref{def:dis-problem} with
\begin{align*}
	a_h(u_h,v_h) &\coloneqq (\Delta_h u_h, \Delta_h v_h)_{L^2(\Omega)} + \s_h(u_h, v_h),\\
	\s_h(u_h,v_h) &\coloneqq \sum_{K \in \Mcal} \sum_{S \in \Sigma(K)} \big(h_S^{-3} (u_K - u_\Sigma, v_K - v_\Sigma)_{L^2(S)}\\
	&\qquad\qquad + h_S^{-1} (\nabla u_K \cdot \nu_S - \beta_\Sigma, \nabla v_K \cdot \nu_S - \gamma_\Sigma)_{L^2(S)}\big),\\
	F(v_h) &\coloneqq (f, v_\Mcal)_{L^2(\Omega)}
\end{align*}
for $u_h = (u_\Mcal, u_\Sigma, \beta_\Sigma), v_h = (v_\Mcal, v_\Sigma, \gamma_\Sigma) \in V_h$.
Here, we use the convention $u_K = u_\Mcal|_K$ for any $K \in \Mcal$.
Note that $a_h$ defines a scalar product in $V_h$ \cite{MuWangYe2014}.
Let $\|\bullet\|_h$ denote the norm induced by $a_h$ and $|\bullet|_\s \coloneqq \sqrt{\s_h(\bullet,\bullet)}$. In particular, \eqref{ineq:coercivity} holds with equality.

\begin{lemma}[equivalence of norms]\label{lem:equiv-norm-WG}
	Any $v_h = (v_\Mcal, v_\Sigma, \gamma_\Sigma) \in V_h$ satisfies
	\begin{align*}
		\|\Delta_h v_h - \Delta_\pw v_\Mcal\| \lesssim |v_h|_\s.
	\end{align*}
	In particular, the norms $\|\bullet\|_h$ and $\big(\|\Delta_\pw v_\Mcal\|^2 + |v_h|_\s^2\big)^{1/2}$ are equivalent.
\end{lemma}
\begin{proof}
	A piecewise integration by parts in \eqref{def:discrete-Laplace-WG}
	implies
	\begin{align*}
		&(\Delta_h v_h - \Delta_\pw v_\Mcal, \phi)_{L^2(\Omega)}\\
		&\qquad = \sum_{K \in \Mcal} \sum_{S \in \Sigma(K)} (\nu_K \cdot \nu_S) \int_S \big((\gamma_\Sigma - \nabla v_K \cdot \nu_S) \phi - (v_\Sigma - v_K) \nabla_\pw \phi \cdot \nu_S\big) \d{s}.
	\end{align*}
	This, the Cauchy inequality, and inverse estimates conclude the assertion. We mention that $\|\bullet\|_h$ and $\big(\|\Delta_\pw v_h\|^2 + |v_h|_\s^2\big)^{1/2}$ are even locally equivalent.
\end{proof}

\subsection{Smoothing operator}
Based on the construction of conforming approximations in \Cref{thm:smoothing}, we obtain the subsequent result.
\begin{lemma}[smoothing operator]\label{lem:smoothing-operator-WG}
	There exists a linear bounded operator $\J_h : V_h \to V$ such that any $v_h = (v_\Mcal, v_\Sigma, \gamma_\Sigma) \in V_h$ satisfies the following properties.
	\begin{enumerate}[wide]
		\item[(i)] (consistent weights) $\Pi_\Mcal^k \J_h v_h = v_\Mcal$, $\Pi_\Sigma^k \J_h v_h = v_\Sigma$, $\Pi_\Sigma^{k-1} \nabla \J_h v_h \cdot \nu_\Sigma = \gamma_\Sigma$.
		\item[(ii)] (approximation property) It holds
		\begin{align*}
			\|h_\Mcal^{-2}(v_\Mcal - \J_h v_h)\| + \|h_\Mcal^{-1} \nabla_\pw(v_\Mcal - \J_h v_h)\| + \|\D_\pw^2(v_\Mcal - \J_h v_h)\| \lesssim |v_h|_\s.
		\end{align*}
	\end{enumerate}
\end{lemma}
\begin{proof}
	Given $v_h = (v_\Mcal, v_\Sigma, \gamma_\Sigma) \in V_h$,
	we employ \Cref{thm:smoothing} with $w = v_\Mcal$, $w_\Mcal = v_\Mcal$, $w_\Sigma = v_\Sigma$, $\delta_\Sigma = \gamma_\Sigma$, $\ell = m = k$, $n = k-1$ to obtain $\J_h v_h = v$ with (i) and, from \Cref{thm:smoothing}(b),
	\begin{align*}
		\|h_\Mcal^{-2}(v_\Mcal - \J_h v_h)\|^2 + \|h_\Mcal^{-1} \nabla_\pw(v_\Mcal - \J_h v_h)\|^2 + \|\D_\pw^2(v_\Mcal - \J_h v_h)\|^2&\\
		\lesssim \sum_{S \in \Sigma} \big(h_S^{-3}\|[v_\Mcal]_S\|_S^2 + h_S^{-1}\|[\nabla_\pw v_\Mcal]_S \cdot \nu_S\|_S^2\big) + |v_h|_\s^2&.
	\end{align*}
	Since $[v_\Mcal]_S = [v_\Mcal - v_\Sigma]_S$ and $[\nabla_\pw v_\Mcal]_S \cdot \nu_S = [\nabla_\pw v_\Mcal \cdot \nu_S - \gamma_\Sigma]_S$ for any $S \in \Sigma$, this and the triangle inequality conclude (ii) and thus, the boundedness of the operator $\J_h$.
\end{proof}
The weights in \Cref{lem:smoothing-operator-WG} provide full discrete consistency as follows.

\begin{lemma}[discrete consistency]\label{lem:d-consistency-WG}
	Any $v_h = (v_\Mcal, v_\Sigma, \gamma_\Sigma) \in V_h$ and $\phi \in P_{k-2}(\Mcal)$ satisfy $(\Delta \J_h v_h - \Delta_h v_h, \phi)_{L^2(\Omega)} = 0$.
\end{lemma}
\begin{proof}
	An integration by parts and the definition of $\Delta_h$ in \eqref{def:discrete-Laplace-WG} prove that $(\Delta \J_h v_h - \Delta_h v_h, \phi)_{L^2(\Omega)}$ is equal to
	\begin{align*}
		&(\J_h v_h - v_\Mcal, \Delta_\pw \phi)_{L^2(\Omega)}\\
		&\quad + \sum_{S \in \Sigma(\Omega)} \int_S \big((\nabla \J_h v_h \cdot \nu_S - \gamma_\Sigma) [\phi]_S - (\J_h v_h - v_\Sigma) [\nabla_\pw \phi]_S \cdot \nu_S\big) \d{s}.
	\end{align*}
	The right-hand side vanishes from \Cref{lem:smoothing-operator-WG}(i).
\end{proof}

\subsection{Interpolation}
Given $v \in V$, let $\mathrm{G}_h v \in P_k(\Mcal)$ denote the Galerkin projection of $v$, i.e., the unique solution to
\begin{align}\label{def:Galerkin}
	\begin{split}
		(\D^2_\pw \mathrm{G}_h v, \D^2_\pw \phi)_{L^2(\Omega)} &= (\D^2 v, \D^2_\pw \phi)_{L^2(\Omega)} \quad\text{for any } \phi \in P_k(\Mcal),\\
		(\mathrm{G}_h v, p)_{L^2(\Omega)} &= (v, p)_{L^2(\Omega)} \quad\text{for any } p \in P_1(\Mcal).
	\end{split}
\end{align}
By design, $\D_\pw^2 (v - \G_h v) \perp \D_\pw^2 P_k(\Mcal)$ in $L^2(\Omega)^{n \times n}$ and, hence,
\begin{align}\label{eq:ba-Galerkin}
	\|\D_\pw^2(v - \G_h v)\| = \min_{\phi \in P_k(\Mcal)} \|\D_\pw^2(v - \phi)\|.
\end{align}
We define the interpolation operator $\I_h : V \to V_h$ by
\begin{align*}
	v \mapsto (\G_h v, \Pi_\Sigma^k v, \Pi_\Sigma^{k-1} \nabla v \cdot \nu_\Sigma) \in V_h.
\end{align*}
The best approximation property of $\G_h$ allows for the subsequent statements.
\begin{lemma}[quasi-optimality of stabilization]\label{lem:quasi-opt-stab-WG}
	Any $v \in V$ satisfies
	\begin{align*}
		|\I_h v|_\s \lesssim \|\D^2_\pw(v - \G_h v)\|.
	\end{align*}
\end{lemma}
\begin{proof}
	Since $L^2$ projections are non expansive,
	\begin{align*}
		|\I_h v|_\s^2 &= \sum_{K \in \Mcal} \sum_{S \in \Sigma(K)} \big(h_S^{-3}\|\Pi_S^k(\G_h v|_K - v)\|_S^2 + h_S^{-1}\|\Pi_S^{k-1}\nabla(\G_h v|_K - v) \cdot \nu_S\|_S^2\big)\\
		&\leq \sum_{K \in \Mcal} \sum_{S \in \Sigma(K)} \big(h_S^{-3}\|v - \G_h v|_K\|_S^2 + h_S^{-1}\|\nabla(v - \G_h v|_K) \cdot \nu_S\|_S^2\big).
	\end{align*}
	The trace and Poincar\'e inequality conclude the proof.
\end{proof}
\begin{remark}[$L^2$ orthogonal projection]
	If we define the interpolation operator $\I_h$ as $v \mapsto (\Pi_\Mcal^k v, \Pi_\Sigma^k v, \Pi_\Sigma^{k-1} \nabla v \cdot \nu_\Sigma) \in V_h$, then $|\I_h v|_\s \lesssim \|\D^2_\pw(v - \Pi_\Mcal^k v)\|$ for any $v \in V$. While $\|\D^2_\pw(v - \G_h v)\| \leq \|\D^2_\pw(v - \Pi_\Mcal^k v)\|$, it is not clear how the constant in the reverse direction behaves on unstructured polytopal meshes.
\end{remark}

\begin{lemma}[quasi-optimality of discrete Laplacian]\label{lem:quasi-opt-dis-Laplace-WG}
	Any $v \in V$ satisfies
	\begin{align*}
		\|\Delta v - \Delta_h \I_h v\| \lesssim \|\D^2_\pw(v - \G_h v)\|.
	\end{align*}
\end{lemma}
\begin{proof}
	The assertion follows from a triangle inequality, \Cref{lem:equiv-norm-WG}, $\|\Delta_\pw(v - \G_h v)\| \lesssim \|\D^2_\pw(v - \G_h v)\|$, and \Cref{lem:quasi-opt-stab-WG}.
\end{proof}

\begin{lemma}[conforming quasi-interpolation]\label{lem:quasi-int-WG}
	Any $v \in V$ satisfies
	\begin{align*}
		\|\Delta(v - \J_h \I_h v)\| = \|\D^2 (v - \J_h \I_h v)\| \lesssim \|\D^2_\pw(v - \G_h v)\|.
	\end{align*}
\end{lemma}
\begin{proof}
	The sum of \Cref{lem:smoothing-operator-WG}(ii) over all cells $K \in \Mcal$ leads to
	\begin{align*}
		\|\D^2_\pw(\G_h v - \J_h \I_h v)\| \lesssim |\I_h v|_\s.
	\end{align*}
	This, a triangle inequality, \eqref{eq:ba-Galerkin}, and \Cref{lem:quasi-opt-stab-WG} conclude the assertion.
\end{proof}

\subsection{Error estimates}
This section establishes the main results \eqref{ineq:main_results}.
\begin{theorem}[a~priori]\label{thm:a-priori-WG}
	It holds
	\begin{align*}
		\|\D^2_\pw(u - u_\Mcal)\| + \|\I_h u - u_h\|_h + |u_h|_\s \lesssim \|\D^2_\pw(u - \G_h u)\| + \osc(f,\Mcal)
	\end{align*}
	with the data oscillation $\osc(f,\Mcal) \coloneqq \|h_\Mcal^2(1 - \Pi_\Mcal^k) f\|$.
\end{theorem}
\begin{proof}
	We abbreviate $e_h = (e_\Mcal, e_\Sigma, \alpha_\Sigma) \coloneqq \I_h u - u_h \in V_h$. 
	From \Cref{lem:smoothing-operator-WG} and the Poincar\'e inequality, we infer
	\begin{align}\label{ineq:pr-a-priori-data-appr}
		F(\J_h e_h) - F_h(e_h) = ((1 - \Pi_\Mcal^k) f, \J_h e_h - e_\Mcal)_{L^2(\Omega)} \lesssim \osc(f,\Mcal) \|e_h\|_{h}.
	\end{align}
	Since $(\Delta_h \I_h u, \Delta_h e_h)_{L^2(\Omega)} = (\Delta_h \I_h u, \Delta \J_h e_h)_{L^2(\Omega)}$ from \Cref{lem:d-consistency-WG},
	\begin{align*}
		a_h(\I_h u, e_h) - a(u,\J_h e_h) = (\Delta_h \I_h u - \Delta u, \Delta \J_h e_h)_{L^2(\Omega)} + \s_h(\I_h u, e_h).
	\end{align*}
	The two previously displayed formula, \Cref{lem:quasi-opt-dis-Laplace-WG}, the continuity of $\J_h$, and the error split \eqref{ineq:err-split} imply
	\begin{align*}
		\|e_h\|^2_h &= a_h(e_h,e_h) = a_h(\I_h u, e_h) - a(u, \J_h e_h) + F(\J_h e_h) - F_h(e_h)\\
		&\leq C\big(\|\D^2_\pw(u - \G_h u)\| + \osc(f,\Mcal)\big)\|e_h\|_h + \s_h(\I_h u, e_h)
	\end{align*}
	with a generic positive constant $C$.
	Since $\s_h(\I_h u, e_h) = |\I_h u|^2_\s/2 + |e_h|^2_\s/2 - |u_h|_\s^2/2$, this and \Cref{lem:quasi-opt-stab-WG} yield
	\begin{align}\label{ineq:pr-a-priori}
		\|e_h\|_h + |u_h|_\s \lesssim \|\D^2_\pw(u - \G_h u)\| + \osc(f,\Mcal).
	\end{align}
	Finally, we deduce from the triangle inequality, \Cref{lem:smoothing-operator-WG}, and \Cref{lem:quasi-int-WG} that
	\begin{align*}
		\|\D_\pw^2(u - u_\Mcal)\| &\leq \|\D^2(u - \J_h \I_h u)\| + \|\D^2(\J_h \I_h u - \J_h u_h)\| + \|\D^2_\pw(\J_h u_h - u_\Mcal)\|\\
		&\lesssim \|\D^2_\pw(u - \G_h u)\| + \|e_h\|_h + |u_h|_\s. 
	\end{align*}
	This and \eqref{ineq:pr-a-priori} conclude the proof.
\end{proof}
\begin{remark}[WG method of \cite{ZhangZhai2015} with reduced polynomial order]
	To extend the analysis of this section to \cite{ZhangZhai2015}, we additionally need
	\begin{align*}
		\|\D^2_\pw \phi\|^2 \lesssim \|\Delta_\pw \phi\|^2 + \sum_{S \in \Sigma} \big(h_S^{-3}\|\Pi_S^{k-1}[\phi]\|^2_S + h_S^{-1}\|[\nabla_\pw \phi]_S \cdot \nu_S\|^2_S\big)
	\end{align*}
	for any $\phi \in P_k(\Mcal)$
	to ensure the boundedness of $\J_h$. Unfortunately, we are unable to prove this discrete Miranda-Talenti typed inequality.
\end{remark}

In the following, we derive lower-order error estimates via Aubin-Nitsche duality arguments with the following elliptic regularity \cite{BlumRannacher1980,Grisvard2011,Dauge1988}. Let $0 < s < 2$ be given. We assume that there exists $0 < \delta \leq 2 - s$ such that any solution $z \in V$ to $\Delta^2 z = G \in H^{-s}(\Omega)$ satisfies $z \in H^{2+\delta}(\Omega)$ with
\begin{align}\label{ineq:elliptic-regularity}
	\|z\|_{H^{2+\delta}(\Omega)} \lesssim \|G\|_{H^{-s}(\Omega)}.
\end{align}
\vspace*{-1.5em}
\begin{theorem}[$H^s$ error]\label{thm:Hs}
	Suppose that \eqref{ineq:elliptic-regularity} holds. Then
	\begin{align*}
		\|u - u_\Mcal\|_{H^s(\Mcal)} + \|u - \J_h u_h\|_{H^s(\Omega)} \lesssim h_\mathrm{max}^{\min\{\delta,k-1\}} \big(\|\D^2_\pw(u - \G_h u)\| + \osc(f,\Mcal)\big)
	\end{align*}
	with the piecewise $H^s$ norm $\|\bullet\|_{H^s(\Mcal)}$ w.r.t.~the mesh $\Mcal$.
\end{theorem}
\begin{proof}
	By Hahn-Banach theorem, there exists $G \in H^{-s}(\Omega)$ with $\|G\|_{H^{-s}(\Omega)} = 1$ and $\|u - \J_h u_h\|_{H^s(\Omega)} = G(u - \J_h u_h)$. 
	Let $z \in V$ denote the solution to $\Delta^2 z = G$. The proof departs from the split
	\begin{align*}
		\|u - \J_h u_h\|_{H^s(\Omega)} = G(u - \J_h u_h) = a(z - \J_h \I_h z, u - \J_h u_h) + a(\J_h \I_h z, u - \J_h u_h).
	\end{align*}
	Since $(\Delta_h \I_h z, \Delta \J_h u_h)_{L^2(\Omega)} = (\Delta_h \I_h z, \Delta_h u_h)_{L^2(\Omega)}$ from \Cref{lem:d-consistency-WG},
	the second term is equal to
	\begin{align*}
		a(\J_h \I_h z, \J_h u_h) = (\Delta \J_h \I_h z - \Delta_h \I_h z, \Delta \J_h u_h)_{L^2(\Omega)} + (\Delta_h \I_h z, \Delta_h u_h)_{L^2(\Omega)}.
	\end{align*}
	The identity $(\Delta_h \I_h z, \Delta_h u_h)_{L^2(\Omega)} = (f, \G_h z)_{L^2(\Omega)} - \s_h(\I_h z, u_h)$ from \eqref{def:dis-problem},
	the two previously displayed formula, and \eqref{def:cont-problem} provide
	\begin{align}\label{eq:pr-Hs-estimates}
		\|u - \J_h u_h\|_{H^s(\Omega)} &= a(z - \J_h \I_h z, u - \J_h u_h) + (f, \J_h\I_h z - \G_h z)_{L^2(\Omega)}\nonumber\\
		&\quad- (\Delta \J_h \I_h z - \Delta_h \I_h z, \Delta \J_h u_h)_{L^2(\Omega)} + \s_h(\I_h z, u_h).
	\end{align}
	The terms on the right-hand side can be bounded one by one as follows. The triangle inequality $\|\Delta(u - \J_h u_h)\| \leq \|\Delta_\pw(u - u_\Mcal)\| + \|\Delta_\pw(u_\Mcal - \J_h u_h)\|$, \Cref{lem:smoothing-operator-WG}(ii), and \Cref{lem:quasi-int-WG} imply
	\begin{align*}
		a(z - \J_h \I_h z, u - \J_h u_h) \lesssim \|\D^2_\pw(z - \G_h z)\| \big(\|\D^2_\pw(u - u_\Mcal)\| + |u_h|_\s\big).
	\end{align*}
	From \Cref{lem:smoothing-operator-WG}, the Poincar\'e inequality, and \Cref{lem:quasi-opt-stab-WG}, we deduce that
	\begin{align*}
		(f, \J_h\I_h z - \G_h z)_{L^2(\Omega)} \lesssim \osc(f,\Mcal)\|\D^2_\pw(z - \G_h z)\|.
	\end{align*}
	Since $(\Delta \J_h \I_h z - \Delta_h \I_h z, \Delta \J_h u_h)_{L^2(\Omega)} = (\Delta \J_h \I_h z - \Delta_h \I_h z, \Delta_\pw (\J_h u_h - u_\Mcal))_{L^2(\Omega)}$ from \Cref{lem:d-consistency-WG}, \Cref{lem:equiv-norm-WG}--\ref{lem:smoothing-operator-WG} and \Cref{lem:quasi-opt-stab-WG} imply
	\begin{align*}
		-(\Delta \J_h \I_h z - \Delta_h \I_h z, \Delta \J_h u_h)_{L^2(\Omega)} \lesssim \|\D^2_\pw(z - \G_h z)\| |u_h|_\s.
	\end{align*}
	In combination with $\s_h(\I_h z, u_h) \leq |\I_h z|_\s |u_h|_\s$, \Cref{lem:quasi-opt-stab-WG}, \Cref{thm:a-priori-WG}, and \eqref{ineq:elliptic-regularity}--\eqref{eq:pr-Hs-estimates} yield
	\begin{align}\label{ineq:pr-Hs-u-Juh}
		\|u - \J_h u_h\| &\lesssim \|\D^2_\pw(z - \G_h z)\|\big(\|\D^2_\pw(u - \G_h u)\| + \osc(f,\Mcal)\big)\nonumber\\
		&\lesssim h^{\min\{\delta,k-1\}}_{\max}\big(\|\D^2_\pw(u - \G_h u)\| + \osc(f,\Mcal)\big).
	\end{align}
	To conclude the proof, we can proceed as in the proof of \cite[Theorem 4.1]{CarstensenNataraj2021}.
	Standard interpolation estimates and scaling arguments imply
	\begin{align*}
		\|u_\Mcal - \J_h u_h\|_{H^s(\mathcal{M})} \lesssim h_\mathrm{max}^{\lceil s \rceil - s}\big(\|h_\Mcal^{-1}(u_\Mcal - \J_h u_h)\|_{H^{\lceil s \rceil - 1}(\Mcal)} + \|u_\Mcal - \J_h u_h\|_{H^{\lceil s \rceil}(\Mcal)}\big).
	\end{align*}
	This, \Cref{lem:smoothing-operator-WG}(ii), and \Cref{thm:a-priori-WG} provide
	\begin{align*}
		\|u_\Mcal - \J_h u_h\|_{H^s(\mathcal{M})} \lesssim h_\mathrm{max}^{2-s}|u_h|_\s \lesssim h_\mathrm{max}^{2-s} \big(\|\nabla_\pw(u - \G_h u)\| + \osc(f,\Mcal)\big).
	\end{align*}
	Since $\delta \leq 2 - s$ by assumption, the assertion follows from this, \eqref{ineq:pr-Hs-u-Juh}, and a triangle inequality.
\end{proof}

\section{Discontinuous Galerkin}\label{sec:DG}
This section derives quasi-optimal error estimates for the symmetric and nonsymmetric interior penalty DG FEM \cite{MozolevskiSueli2003,GeorgoulisHouston2009}.
\subsection{Discrete problem}
We consider the discrete problem \eqref{def:dis-problem} with the ansatz space $V_h \coloneqq P_k(\Mcal)$, $k \geq 2$, and
\begin{align*}
	a_h(u_h, v_h) &\coloneqq (\Delta_\pw u_h, \Delta_\pw v_h)_{L^2(\Omega)} + \sigma\s_h(u_h, v_h)\\
	&\quad + \sum_{S \in \Sigma} \int_S \big(\theta[u_h]_S\{\nabla_\pw \Delta_\pw v_h\}_S \cdot \nu_S + [v_h]_S\{\nabla_\pw \Delta_\pw u_h\}_S \cdot \nu_S\\
	&\quad\quad\quad - \theta\{\Delta_\pw v_h\}_S[\nabla_\pw u_h]_S \cdot \nu_S - \{\Delta_\pw u_h\}_S[\nabla_\pw v_h]_S \cdot \nu_S\big) \d{s},\\
	\s_h(u_h,v_h) &\coloneqq \sum_{S \in \Sigma} \int_S \big(h_S^{-3}[u_h]_S[v_h]_S + h_S^{-1}[\nabla_\pw u_h]_S \cdot \nu_S [\nabla_\pw v_h]_S \cdot \nu_S\big) \d{s},\\
	F_h(v_h) &\coloneqq (f,v_h)_{L^2(\Omega)}
\end{align*}
for $u_h, v_h \in V_h$ and a positive stabilization parameter $\sigma > 0$.
The choice $\theta = 1$ leads to the symmetric interior penalty (SIP) DG FEM \cite{GeorgoulisHouston2009}, while, for $\theta = -1$, we obtain the nonsymmetric interior penalty (NIP) DG FEM of \cite{MozolevskiSueli2003}. The coercivity \eqref{ineq:coercivity} is satisfied for sufficiently large stabilization parameter $\sigma$ if $\theta = 1$ and unconditionally if $\theta = -1$ with the discrete norm
\begin{align*}
	\|v_h\|_h^2 \coloneqq \|\Delta_\pw v_h\|^2 + |v_h|_\s^2 \quad\text{for any } v_h \in V_h.
\end{align*}
The discrete Laplacian $\Delta_h v_h \in P_{k-2}(\Mcal)$ of $v_h \in V_h$ is the unique solution to
\begin{align}\label{def:Delta-h}
	&(\Delta_h v_h, \phi)_{L^2(\Omega)} = (v_h, \Delta_\pw \phi)_{L^2(\Omega)}\nonumber\\
	&\qquad + \sum_{S \in \Sigma(\Omega)} \int_S \big(\{\nabla_\pw v_h\}_S \cdot \nu_S [\phi]_S - \{v_h\}_S [\nabla_\pw \phi]_S \cdot \nu_S\big) \d{s}.
\end{align}
An integration by parts provides
\begin{align}\label{eq:Delta-h-ibp}
	&(\Delta_h v_h, \phi)_{L^2(\Omega)} = (\Delta_\pw v_h, \phi)_{L^2(\Omega)}\nonumber\\
	&\qquad - \sum_{S \in \Sigma}
	\int_S \big([\nabla_\pw v_h]_S \cdot \nu_S \{\phi\}_S - [v_h]_S \{\nabla_\pw \phi\}_S \cdot \nu_S\big) \d{s}. 
\end{align}
and so, elementary calculations lead to
\begin{align*}
	&(\Delta_h u_h, \Delta_h v_h)_{L^2(\Omega)} = (\Delta_\pw u_h, \Delta_\pw v_h)_{L^2(\Omega)}\\
	&\qquad- \sum_{S \in \Sigma} \int_S \big([\nabla_\pw u_h]_S \cdot \nu_S \{\Delta_h v_h\}_S - [u_h]_S \{\nabla_\pw \Delta_h v_h\}_S \cdot \nu_S\\
	&\qquad\qquad\qquad + [\nabla_\pw v_h]_S \cdot \nu_S \{\Delta_\pw u_h\}_S - [v_h]_S \{\nabla_\pw \Delta_\pw u_h\}_S \cdot \nu_S\big) \d{s}.
\end{align*}
Therefore, $a_h$ can be rewritten as
\begin{align*}
	a_h(u_h,v_h) = (\Delta_h u_h, \Delta_h v_h)_{L^2(\Omega)} + b_h(u_h,v_h) + \sigma\s_h(u_h,v_h)
\end{align*}
with the bilinear form
\begin{align}\label{def:b-h-4th}
	b_h(u_h,v_h) &\coloneqq \sum_{S \in \Sigma} \int_S \big(\theta[u_h]_S\{\nabla_\pw \Delta_\pw v_h\}_S \cdot \nu_S - [u_h]_S\{\nabla_\pw \Delta_h v_h\}_S \cdot \nu_S\nonumber\\
	&\qquad\qquad - \theta\{\Delta_\pw v_h\}_S[\nabla_\pw u_h]_S \cdot \nu_S - \{\Delta_h v_h\}_S [\nabla_\pw u_h]_S \cdot \nu_S\big) \d{s}.
\end{align}
This equivalent formulation is the crucial ingredient that allows for an extension of the analysis of \Cref{sec:WG} to the DG methodology.
\begin{lemma}[equivalence of norms]\label{lem:equiv-norm-DG}
	Any $v_h \in V_h$ satisfies
	\begin{align*}
		\|\Delta_h v_h - \Delta_\pw v_h\| \lesssim |v_h|_\s.
	\end{align*}
	In particular, the norms $\|\bullet\|_h$ and $\big(\|\Delta_h \bullet\|^2 + |\bullet|_\s^2\big)^{1/2}$ are equivalent.
\end{lemma}
\begin{proof}
	The assertion follows immediately from \eqref{eq:Delta-h-ibp}, the Cauchy, and the discrete trace inequality.
\end{proof}
\subsection{Smoothing operator}
Based on the discrete Laplacian $\Delta_h$ in \eqref{def:Delta-h}, we propose the following smoothing operator for discrete consistency.
\begin{lemma}[smoothing operator]\label{lem:smoothing-operator-DG}
	There exists a linear bounded operator $\J_h : V_h \to V$ such that any $v_h \in V_h$ satisfies the following properties.
	\begin{enumerate}[wide]
		\item[(i)] (consistent weights) 
		$\Pi_\Mcal^k \J_h v_h = v_h$, $\Pi_\Sigma^k \J_h v_h = \{v_h\}_\Sigma$, $\Pi_\Sigma^{k-1} \nabla \J_h v_h \cdot \nu_\Sigma = \{\nabla_\pw v_h\}_\Sigma \cdot \nu_\Sigma$.
		\item[(ii)] (local approximation property) It holds
		\begin{align*}
			\|h_\Mcal^{-2}(v_\Mcal - \J_h v_h)\| + \|h_\Mcal \nabla_\pw(v_\Mcal - \J_h v_h)\| + \|\D_\pw^2(v_\Mcal - \J_h v_h)\| \lesssim |v_h|_\s.
		\end{align*}
	\end{enumerate}
\end{lemma}
\begin{proof}
	The assertion follows from \Cref{thm:smoothing} with $w = v_h$, $w_\Mcal = v_h$, $w_\Sigma = \{v_h\}_\Sigma$, $\delta_\Sigma = \{\nabla_\pw v_h\}_\Sigma \cdot \nu_\Sigma$, $\ell = m = k$, and $n = k - 1$.
\end{proof}
The weights in \Cref{lem:smoothing-operator-DG}(i) imply the following.
\begin{lemma}[discrete consistency]\label{lem:d-consistency-DG}
	Any $v_h \in V_h$ and $\phi \in P_{k-2}(\Mcal)$ satisfy
	\begin{align*}
		(\Delta \J_h v_h - \Delta_h v_h, \phi)_{L^2(\Omega)} = 0.
	\end{align*}
\end{lemma}
\begin{proof}
	An integration by parts and the definition of $\Delta_h$ in \eqref{def:Delta-h} prove that $(\Delta \J_h v_h - \Delta_h v_h, \phi)_{L^2(\Omega)}$ is equal to
	\begin{align*}
		&(\J_h v_h - v_h, \Delta_\pw \phi)_{L^2(\Omega)}\\
		&\quad + \sum_{S \in \Sigma(\Omega)} \int_S \big((\nabla \J_h v_h - \{\nabla_\pw v_h\}_S) \cdot \nu_S [\phi]_S - (\J_h v_h - \{v_h\}_S) [\nabla_\pw \phi]_S \cdot \nu_S\big) \d{s}.
	\end{align*}
	This vanishes due to
	the weights prescribed in \Cref{lem:smoothing-operator-DG}(i).
\end{proof}

\subsection{Interpolation}
We consider the interpolation $\I_h \coloneqq \G_h$ with the Galerkin projection from \eqref{def:Galerkin}.
Together with the smoothing operator defined in \Cref{lem:smoothing-operator-DG}, it is straight-forward to verify the following preliminary results corresponding to \Cref{lem:quasi-opt-stab-WG}--\ref{lem:quasi-int-WG}.

\begin{lemma}[quasi-optimality of stabilization]\label{lem:quasi-opt-stab-DG}
	Any $v \in V$ satisfies
	\begin{align*}
		|\I_h v|_\s \lesssim \|\D_\pw^2(v - \I_h v)\|.
	\end{align*}
\end{lemma}
\begin{proof}
	Since $[v]_S$ and $[\nabla v]_S$ vanish along any $S \in \Sigma$, $[\I_h v]_S = [\I_h v - v]_S$ and $[\nabla_\pw \I_h v]_S = [\nabla_\pw(\I_h v - v)]_S$. This, the trace, and Poincar\'e inequality conclude the proof.
\end{proof}

\begin{lemma}[quasi-optimality of discrete Laplacian]\label{lem:q-opt-dis-Laplace-DG}
	Any $v \in V$ satisfies
	\begin{align*}
		\|\Delta v - \Delta_h \I_h v\| \lesssim \|\D_\pw^2(v - \I_h v)\|.
	\end{align*}
\end{lemma}
\begin{proof}
	The assertion follows from a triangle inequality, \Cref{lem:equiv-norm-DG}, \Cref{lem:quasi-opt-stab-DG}, and $\|\Delta_\pw(v - \I_h v)\| \lesssim \|\D^2_\pw(v - \I_h v)\|$.
\end{proof}

\begin{lemma}[conforming quasi-interpolation]\label{lem:q-interpolation-DG}
	Any $v \in V$ satisfies
	\begin{align*}
		\|\Delta(v - \J_h \I_h v)\| = \|\D^2(v - \J_h \I_h v)\| \lesssim \|\D_\pw^2(v - \I_h v)\|.
	\end{align*}
\end{lemma}
\begin{proof}
	\Cref{lem:smoothing-operator-DG}(ii), \Cref{lem:quasi-opt-stab-DG}, and a triangle inequality imply the assertion.
\end{proof}

The additional term $b_h$ in DG methods can be controlled as follows.
\begin{lemma}[quasi-optimal remainder]\label{lem:q-opt-remainder-DG}
	Any $v \in V$ and $w_h \in V_h$ satisfy
	\begin{align*}
		|b_h(\I_h v, w_h)| \lesssim \|w_h\|_h\|\D^2_\pw(v - \I_h v)\|.
	\end{align*}
\end{lemma}
\begin{proof}
	The Cauchy inequality in \eqref{def:b-h-4th} proves
	\begin{align*}
		|b_h(\I_h v, w_h)| \leq |\I_h v|_{\s} \Big(\sum_{S \in \Sigma} h_S^3\|\{\nabla_\pw \Delta_h w_h - \theta\nabla_\pw \Delta_\pw w_h\}_S\|_S^2&\\
		+ h_S\|\{\Delta_\pw w_h - \theta \Delta_h w_h\}_S\|^2_S&\Big)^{1/2}.
	\end{align*}
	The triangle inequality, the discrete trace inequality, the inverse estimate, the boundedness of $\Delta_h$ from \Cref{lem:equiv-norm-DG}, and \Cref{lem:quasi-opt-stab-DG} conclude the assertion.
\end{proof}

\subsection{Error estimates}
The following results correspond to \eqref{ineq:main_results} and are the analogue of \Cref{thm:a-priori-WG} and \Cref{thm:Hs}.
\begin{theorem}[a~priori]\label{thm:a-priori-DG}
	It holds
	\begin{align*}
		\|\D^2_\pw(u - u_h)\| + \|\I_h u - u_h\|_h + |u_h|_\s \lesssim \|\D^2_\pw(u - \I_h u)\| + \osc(f,\Mcal).
	\end{align*}
\end{theorem}
\begin{proof}
	Since $(\Delta_h \I_h u, \Delta_h e_h)_{L^2(\Omega)} = (\Delta_h \I_h u, \Delta \J_h e_h)_{L^2(\Omega)}$ from \Cref{lem:d-consistency-DG}, we infer
	\begin{align*}
		a_h(\I_h u, e_h) - a(u,\J_h e_h) = (\Delta_h \I_h u - \Delta u, \Delta \J_h e_h)_{L^2(\Omega)} + b_h(\I_h u, e_h) + \sigma\s_h(\I_h u, e_h).
	\end{align*}
	This, \eqref{ineq:err-split}, \eqref{ineq:pr-a-priori-data-appr}, \Cref{lem:q-opt-dis-Laplace-DG}, \Cref{lem:q-opt-remainder-DG}, and the continuity of $\J_h$ imply
	\begin{align*}
		\|e_h\|^2_h \lesssim \big(\|\D^2_\pw(u - \I_h u)\| + \osc(f,\Mcal) + |\I_h u|_\s\big)\|e_h\|_h.
	\end{align*}
	Dividing by $\|e_h\|_h$ and employing \Cref{lem:quasi-opt-stab-DG} provide
	\begin{align*}
		\|e_h\|_h \lesssim \|\D^2_\pw(u - \I_h u)\| + \osc(f,\Mcal).
	\end{align*}
	This,
	the triangle inequality $|u_h|_\s \leq |e|_\s + |\I_h u|_\s$, and \Cref{lem:quasi-opt-stab-DG} prove 
	\begin{align*}
		|u_h|_\s \lesssim \|\D^2_\pw(u - \I_h u)\| + \osc(f,\Mcal).
	\end{align*}
	The two previously displayed formula also controls $\|\D_\pw^2(u - u_h)\|$ because
	\begin{align*}
		\|\D_\pw^2(u - u_h)\| &\leq \|\D^2(u - \J_h \I_h u)\| + \|\D^2(\J_h \I_h u - \J_h u_h)\| + \|\D^2_\pw(\J_h u_h - u_h)\|\\
		&\lesssim \|\D^2_\pw(u - \I_h u)\| + \|e_h\|_h + |u_h|_\s
	\end{align*}
	from a triangle inequality, \Cref{lem:smoothing-operator-DG}, and \Cref{lem:q-interpolation-DG}.
\end{proof}

\begin{remark}[medius analysis of \cite{Gudi2010}]
	Quasi-optimal error estimates for the SIP DG FEM on simplicial meshes has been established in \cite[Section 4.3]{Gudi2010}.
\end{remark}

The following result states lower-order error estimates for the SIP DG FEM.

\begin{theorem}[$H^s$ error]\label{thm:Hs-DG}
	Suppose that $\theta = 1$ (SIP) and \eqref{ineq:elliptic-regularity} hold. Then
	\begin{align*}
		\|u - u_h\|_{H^s(\Mcal)} + \|u - \J_h u_h\|_{H^s(\Omega)} \lesssim h_\mathrm{max}^{\min\{\delta,k-1\}}\big(\|\D^2_\pw(u - \I_h u)\| + \osc(f,\Mcal)\big).
	\end{align*}
\end{theorem}
\begin{proof}
	The assertion follows along the lines of the proof of \Cref{thm:Hs}, but requires a bound on $-b_h(u_h, \I_h z)$ that additionally arises on the right-hand side of \eqref{eq:pr-Hs-estimates}. The Cauchy inequality proves that $|b_h(u_h, \I_h z)|$ is bounded by
	\begin{align*}
		|u_h|_\s \big(\sum_{S \in \Sigma} h_S^3\|\{\nabla_\pw(\Delta_\pw - \Delta_h) \I_h z\}_S \cdot \nu_S\|_S^2 + h_S\|\{(\Delta_\pw - \Delta_h) \I_h z\}_S\|_S^2\big).
	\end{align*}
	The triangle inequality, the discrete trace inequality, and inverse estimates imply $|b_h(u_h, \I_h z)| \lesssim |u_h|_\s\|(\Delta_\pw - \Delta_h) \I_h z\|$. This, \Cref{lem:equiv-norm-DG}, and \Cref{lem:quasi-opt-stab-DG} yield
	\begin{align*}
		|b_h(u_h, \I_h z)| \lesssim |u_h|_\s \|\D^2_\pw(z - \I_h z)\|.
	\end{align*} 
	The right-hand side is quasi-optimal by \Cref{thm:a-priori-DG}.
	Further details on the proof of \Cref{thm:Hs-DG} are omitted for the sake of brevity.
\end{proof}

\begin{remark}[Lower-order error estimates for NIP DG]
	For the NIP DG FEM with $\theta = -1$, the term $|b_h(u_h, \I_h z)| \lesssim |u_h|_\s \|\D^2 z\|$ displays reduced convergence rates, leading to
	\begin{align*}
		\|u - u_h\|_{H^s(\Mcal)} + \|u - \J_h u_h\|_{H^s(\Omega)} \lesssim \|\nabla_\pw(u - \I_h u)\| + \osc(f,\Mcal).
	\end{align*}
	This is consistent with known theoretical results from \cite{ArnoldBrezziCockburnMarini2001} for smooth solutions.
\end{remark}

\section{Hybrid high-order}\label{sec:HHO}
This section derives quasi-optimal error estimates for the HHO FEM of \cite{DongErn2022}. For the sake of brevity, we focus on the three dimensional case $d = 3$ and mention that the arguments carry over to the two dimensional case $d = 2$ as well.

\subsection{Discrete problem}
We consider the ansatz space $V_h \coloneqq P_k(\Mcal) \times P_k(\Sigma(\Omega)) \times P_{k-2}(\Sigma(\Omega))$ with $k \geq 2$. Given $v_h = (v_\Mcal, v_\Sigma, \gamma_\Sigma) \in V_h$, the reconstruction $\mathcal{R}_h v_h \in P_k(\Mcal)$ of $v_h$ is the unique solution to $\Pi_\Mcal^1 \mathcal{R}_h v_h = \Pi_\Mcal^1 v_\Mcal$ and, for any $\phi \in P_k(\Mcal)$,
\begin{align}\label{def:rec-HHO}
	(\D^2_\pw \mathcal{R}_h v_h&, \D^2_\pw \phi)_{L^2(\Omega)} = (v_\Mcal, \Delta^2_\pw \phi)_{L^2(\Omega)}\nonumber\\
	&- \sum_{S \in \Sigma(\Omega)} \int_S \big(v_\Sigma  [\partial_\nu \Delta_\pw \phi]_S - \gamma_\Sigma [\partial_{\nu\nu} \phi]_S - \partial_t v_\Sigma [\partial_{\nu t} \phi]_S\big) \d{s}.
\end{align}
Here and throughout the remaining parts of this section, $\partial_\nu$ and $\partial_t$ denote the normal and tangential derivatives.
The discrete problem seeks the solution $u_h = (u_\Mcal, u_\Sigma, \beta_\Sigma) \in V_h$ to \eqref{def:dis-problem} with
\begin{align*}
	a_h(u_h,v_h) &\coloneqq (\D^2_\pw \mathcal{R}_h u_h, \D^2_\pw \mathcal{R}_h v_h)_{L^2(\Omega)} + \s_h(u_h, v_h),\\
	\s_h(u_h,v_h) &\coloneqq \sum_{K \in \Mcal} \sum_{S \in \Sigma(K)} h_S^{-3} (u_K - u_\Sigma, v_K - v_\Sigma)_{L^2(S)}\\
	&\qquad\qquad + h_S^{-1} (\Pi_S^{k-2}(\nabla u_K \cdot \nu_S - \beta_\Sigma), \Pi_S^{k-2}(\nabla v_K \cdot \nu_S - \gamma_\Sigma))_{L^2(S)},\\
	F_h(v_h) &\coloneqq (f, v_\Mcal)_{L^2(\Omega)}
\end{align*}
for $v_h = (v_\Mcal, v_\Sigma, \gamma_\Sigma) \in V_h$.
The discrete bilinear form $a_h$ defines a scalar product in $V_h$ \cite{DongErn2022} and, therefore, \eqref{ineq:coercivity} holds with equality for the induced norm $\|\bullet\|_h \coloneqq \sqrt{a_h(\bullet,\bullet)}$.
By $|\bullet|_\s \coloneqq \sqrt{\s_h(\bullet,\bullet)}$, we denote the seminorm induced by $\s_h$.

\begin{lemma}[equivalence of norms]\label{lem:equiv-norms-HHO}
	Any $v_h = (v_\Mcal, v_\Sigma, \gamma_\Sigma) \in V_h$ satisfies
	\begin{align*}
		\|\D^2_\pw(\mathcal{R}_h v_h - v_\Mcal)\| \lesssim |v_h|_\s.
	\end{align*}
	In particular, the norms $\|\bullet\|_h$ and $\big(\|\D^2_\pw v_\Mcal\|^2 + |v_h|_\s^2\big)^{1/2}$ are equivalent.
\end{lemma}
\begin{proof}
	The assertion follows from a piecewise integration by parts formula \cite[Eq.~(19)]{DongErn2022}, the discrete trace inequality, and inverse estimates, cf.~also \cite[Lemma 4.1]{DongErn2022}.
\end{proof}

\subsection{Smoothing operator}
We propose the following smoother based on \Cref{thm:smoothing} for discrete consistency.
\begin{lemma}[smoothing operator]\label{lem:smoothing-operator-HHO}
	There exists a linear bounded operator $\J_h : V_h \to V$ such that any $v_h = (v_\Mcal, v_\Sigma, \gamma_\Sigma) \in V_h$ satisfies the following properties.
	\begin{enumerate}[wide]
		\item[(i)] (consistent weights) $\Pi_\Mcal^k \J_h v_h = v_\Mcal$, $\Pi_\Sigma^k \J_h v_h = v_\Sigma$, $\Pi_\Sigma^{k-2} \nabla \J_h v_h \cdot \nu_\Sigma = \gamma_\Sigma$.
		\item[(ii)] (approximation property) It holds
		\begin{align*}
			\|h_\Mcal^{-2}(v_\Mcal - \J_h v_h)\|^2 + \|h_\Mcal^{-1} \nabla_\pw(v_\Mcal - \J_h v_h)\|^2 + \|\D_\pw^2(v_\Mcal - \J_h v_h)\|^2&\\
			\lesssim \sum_{S \in \Sigma} h_S^{-1}\|[\nabla_\pw v_\Mcal]_S \cdot \nu_S\|^2_S + |v_h|_\s^2&.
		\end{align*}
	\end{enumerate}
\end{lemma}
\begin{proof}
	The properties (i)--(ii) follow from \Cref{thm:smoothing} with $w = v_\Mcal$, $w_\Mcal = v_\Mcal$, $w_\Sigma = v_\Sigma$, $\delta_\Sigma = \gamma_\Sigma$, $\ell = m = k$, $n = k-2$, and $\|[v_\Mcal]_S\|_S = \|[v_\Mcal - v_\Sigma]_S\|_S$ for any $S \in \Sigma$. To prove the boundedness of $\J_h$, we employ the Poincar\'e inequality to infer, for any $S \in \Sigma$, that
	\begin{align*}
		\|[\nabla_\pw v_\Mcal]_S \cdot \nu_S\|_S^2 &= \|(1 - \Pi_S^0)[\nabla_\pw v_\Mcal]_S \cdot \nu_S\|_S^2 + \|\Pi_S^0 [\nabla_\pw v_\Mcal]_S \cdot \nu_S\|_S^2\\
		&\lesssim h_S^2\|[\partial_{\nu t} v_\Mcal]_S\|^2_S + \|\Pi_S^{k-2} [\nabla_\pw v_\Mcal]_S \cdot \nu_S\|^2_S.
	\end{align*}
	The second term on the right hand side is equal to $\|\Pi_S^{k-2} [\nabla_\pw v_\Mcal \cdot \nu_S - \gamma_\Sigma]_S\|_S^2$ and is controlled by the stabilization via a triangle inequality. Therefore, we deduce from (ii) and the discrete trace inequality that
	\begin{align*}
		\|\D_\pw^2(v_\Mcal - \J_h v_h)\| \lesssim \|\D_\pw^2 v_\Mcal\| + |v_h|_\s.
	\end{align*}
	This and \Cref{lem:equiv-norms-HHO} conclude the assertion.
\end{proof}
In contrast to the numerical methods analyzed so far, the weights in \Cref{lem:smoothing-operator-HHO}(i) do not allow for full discrete consistency, but rather a weaker version of it.
For lowest-order methods, this is also observed in \cite{GallistlTran26} to derive a~priori error estimates for Discrete Kirchhoff elements.

\begin{lemma}[discrete consistency]\label{lem:d-consistency-HHO}
	Any $v_h = (v_\Mcal, v_\Sigma, \gamma_\Sigma) \in V_h$, $w_h \in V_h$, and $\phi \in P_k(\Mcal)$ satisfy
	\begin{align*}
		|(\D^2_\pw (\mathcal{R}_h v_h - \J_h v_h), \D^2_\pw \phi)_{L^2(\Omega)}| \lesssim \|\D^2_\pw(\phi - \J_h w_h)\|\big(|v_h|_\s + \|\D^2_\pw(v_\Mcal - \J_h v_h)\|\big).
	\end{align*}
\end{lemma}
\begin{proof}
	From an integration by parts and the definition of $\mathcal{R}_h$ in \eqref{def:rec-HHO}, we deduce that $(\D^2_\pw (\mathcal{R}_h v_h - \J_h v_h), \D^2_\pw \phi)_{L^2(\Omega)}$ is equal to
	\begin{align*}
		(v_\Mcal - \J_h v_h, \Delta^2_\pw &\phi)_{L^2(\Omega)}
		- \sum_{S \in \Sigma(\Omega)} \int_S \big((v_\Sigma - \J_h v_h)  [\partial_\nu \Delta_\pw \phi]_S\\
		& - (\gamma_\Sigma - \partial_\nu \J_h v_h) [\partial_{\nu\nu} \phi]_S - \partial_t (v_\Sigma - \J_h v_h)  [\partial_{\nu t} \phi]_S\big) \d{s}.
	\end{align*}
	The weights in \Cref{lem:smoothing-operator-HHO}(i) show that the first three terms vanish. Hence,
	\begin{align*}
		(\D^2_\pw (\mathcal{R}_h v_h - \J_h v_h), \D^2_\pw \phi)_{L^2(\Omega)} = -\sum_{S \in \Sigma(\Omega)} \int_S \partial_t (v_\Sigma - \J_h v_h)  [\partial_{\nu t} \phi]_S \d{s}.
	\end{align*}
	Since $[\partial_{\nu t} \J_h w_h]_S$ vanishes along any face $S \in \Sigma(\Omega)$,
	the discrete trace inequality shows
	\begin{align*}
		|(\D^2_\pw (\mathcal{R}_h v_h - \J_h v_h), \D^2_\pw \phi)_{L^2(\Omega)}| \leq \sum_{S \in \Sigma} \|\partial_t(v_\Sigma - \J_h v_h)\|_S \|[\partial_{\nu t}(\phi - \J_h w_h)]_S\|_S&\\
		\lesssim \|\D^2_\pw(\phi - \J_h w_h)\| \Big(\sum_{S \in \Sigma} h_S^{-1}\|\partial_t(v_\Sigma - \J_h v_h)\|_S^2\Big)^{1/2}&.
	\end{align*}
	Inverse estimates, the triangle inequality, and \Cref{lem:smoothing-operator-HHO} provide
	\begin{align*}
		h_S^{-1}\|\partial_t(v_\Sigma - \J_h v_h)\|_S^2 &\lesssim h_S^{-3}\|v_\Sigma - \J_h v_h\|_S^2\\
		&\lesssim h_S^{-3}\big(\|v_\Sigma - \{v_\Mcal\}_S\|_S^2 + \|\{v_\Mcal\}_S - \J_h v_h\|_S^2\big).
	\end{align*}
	The combination of the two previously displayed formula with the trace and Poincar\'e inequality concludes the proof.
\end{proof}

\subsection{Interpolation}
We consider the interpolation operator
$\I_h : V \to V_h$,
\begin{align*}
	v \mapsto (\G_h v, \Pi_\Sigma^k v, \Pi_\Sigma^{k-2} \nabla v \cdot \nu_\Sigma) \in V_h
\end{align*}
with the Galerkin projection $\G_h$ from \eqref{def:Galerkin}. The following preliminary results correspond to \Cref{lem:quasi-opt-stab-WG}--\ref{lem:quasi-int-WG}.
\begin{lemma}[quasi-optimality of stabilization]\label{lem:quasi-opt-stab-HHO}
	Any $v \in V$ satisfies
	\begin{align*}
		|\I_h v|_\s \lesssim \|\D^2_\pw(v - \G_h v)\|.
	\end{align*}
\end{lemma}
\begin{proof}
	The assertion follows along the lines of the proof of \Cref{lem:quasi-opt-stab-WG}.
\end{proof}

\begin{lemma}[quasi-optimality of reconstruction operator]\label{lem:quasi-opt-rec-HHO}
	Any $v \in V$ satisfies
	\begin{align*}
		\|\D^2_\pw(v - \mathcal{R}_h \I_h v)\| \lesssim \|\D^2_\pw(v - \G_h v)\|.
	\end{align*}
\end{lemma}
\begin{proof}
	A triangle inequality, \Cref{lem:equiv-norms-HHO}, and \Cref{lem:quasi-opt-stab-HHO} imply the assertion.
\end{proof}

\begin{lemma}[conforming quasi-interpolation]\label{lem:quasi-int-HHO}
	Any $v \in V$ satisfies
	\begin{align*}
		\|\D^2 (v - \J_h \I_h v)\| \lesssim \|\D^2_\pw(v - \G_h v)\|.
	\end{align*}
\end{lemma}
\begin{proof}
	The sum of \Cref{lem:smoothing-operator-HHO}(i) over all cells $K \in \Mcal$ leads to
	\begin{align*}
		\|\D^2_\pw(\G_h v - \J_h \I_h v)\|^2 \lesssim \sum_{S \in \Sigma} h_S^{-1} \|[\nabla_\pw \G_h v]_S \cdot \nu_S\|_S^2 + |\I_h v|_\s^2.
	\end{align*}
	Since $[\nabla v]_S \cdot \nu_S$ vanishes along all faces $S \in \Sigma$, the triangle, trace, and Poincar\'e inequality provide $\sum_{S \in \Sigma} h_S^{-1} \|[\nabla_\pw \G_h v]_S \cdot \nu_S\|_S^2 \lesssim \|\D^2_\pw(v - \G_h v)\|^2$. This, the previously displayed formula, and \Cref{lem:quasi-opt-stab-HHO} conclude the proof.
\end{proof}

\subsection{Error estimates}
We state the corresponding version of \eqref{ineq:main_results}.
\begin{theorem}[a~priori]\label{thm:a-priori-HHO}
	It holds
	\begin{align*}
		\|\D^2_\pw(u - u_\Mcal)\| + \|\I_h u - u_h\|_h + |u_h|_\s \lesssim \|\D^2_\pw(u - \G_h u)\| + \osc(f,\Mcal).
	\end{align*}
\end{theorem}
\begin{proof}
	We abbreviate $e_h = (e_\Mcal, e_\Sigma, \alpha_\Sigma) \coloneqq \I_h u - u_h \in V_h$.
	From \Cref{lem:d-consistency-HHO} and the continuity of $\J_h$, we infer that
	\begin{align*}
		(\D^2_\pw \mathcal{R}_h \I_h u, \D^2_\pw (\mathcal{R}_h \I_h e_h - \J_h e_h))_{L^2(\Omega)} \lesssim \|\D^2_\pw(\mathcal{R}_h \I_h u - \J_h \I_h u)\|\|e_h\|_h.
	\end{align*}
	This, a triangle inequality, and \Cref{lem:quasi-opt-rec-HHO}--\Cref{lem:quasi-int-HHO} imply
	\begin{align*}
		(\D^2_\pw \mathcal{R}_h \I_h u, \D^2_\pw \mathcal{R}_h (\I_h e_h - \J_h e_h))_{L^2(\Omega)} \lesssim \|\D^2_\pw(u - \G_h u)\|\|e_h\|_h.
	\end{align*}
	Elementary algebra provides the identity
	\begin{align*}
		a_h(\I_h u, e_h) - a(u,\J_h e_h) = (\D^2_\pw \mathcal{R}_h \I_h u, \D^2_\pw \mathcal{R}_h (\I_h e_h - \J_h e_h))_{L^2(\Omega)}&\\
		+ (\D^2_\pw (\mathcal{R}_h \I_h u - u), \D^2 \J_h e_h)_{L^2(\Omega)} + \s_h(\I_h u, e_h)&.
	\end{align*}
	We combine the two previously displayed formula with \Cref{lem:quasi-opt-rec-HHO}, the continuity of $\J_h$, \eqref{ineq:err-split}, and \eqref{ineq:pr-a-priori-data-appr} to deduce
	\begin{align*}
		\|e_h\|^2_h &= a_h(e_h,e_h) = a_h(\I_h u, e_h) - a(u, \J_h e_h) + F(\J_h e_h) - F_h(e_h)\\
		&\leq C\big(\|\D^2_\pw(u - \G_h u)\| + \osc(f,\Mcal)\big)\|e_h\|_h + \s_h(\I_h u, e_h)
	\end{align*}
	for some generic positive constant $C$.
	The remaining proof can follow that of \Cref{thm:a-priori-WG} and is omitted for the sake of brevity.
\end{proof}

\begin{theorem}[$H^s$ error]\label{thm:Hs-HHO}
	Suppose that \eqref{ineq:elliptic-regularity} holds. Then
	\begin{align*}
		\|u - u_\Mcal\|_{H^s(\Mcal)} + \|u - \J_h u_h\|_{H^s(\Omega)} \lesssim h_\mathrm{max}^{\min\{\delta,k-1\}} \big(\|\D^2_\pw(u - \G_h u)\| + \osc(f,\Mcal)\big).
	\end{align*}
\end{theorem}
\begin{proof}
	Let $G \in H^{-s}(\Omega)$ with $\|G\|_{H^{-s}(\Omega)} = 1$ and $\|u - \J_h u_h\|_{H^s(\Omega)} = G(u - \J_h u_h)$. Furthermore,
	$z \in V$ denotes the solution to $\Delta^2 z = G$. Recall the split
	\begin{align*}
		\|u - \J_h u_h\|_{H^s(\Omega)} = G(u - \J_h u_h) = a(z - \J_h \I_h z, u - \J_h u_h) + a(\J_h \I_h z, u - \J_h u_h)
	\end{align*}
	from the proof of \Cref{thm:Hs}. 
	Elementary algebra shows
	\begin{align*}
		&a(\J_h \I_h z, \J_h u_h) = (\D^2_\pw(\J_h z - \mathcal{R}_h \I_h z), \D^2 \J_h u_h)_{L^2(\Omega)}\\
		&\qquad + (\D^2_\pw \mathcal{R}_h \I_h z, \D^2_\pw(\J_h u_h - \mathcal{R}_h u_h))_{L^2(\Omega)} + (\D^2_\pw \mathcal{R}_h \I_h z, \D^2_\pw \mathcal{R}_h u_h)_{L^2(\Omega)}.
	\end{align*}
	The two previously displayed formula and $$a(\J_h \I_h z, u) - (\D^2_\pw \mathcal{R}_h \I_h z, \D^2_\pw \mathcal{R}_h u_h)_{L^2(\Omega)} = (f, \J_h \I_h z - \G_h z)_{L^2(\Omega)} + \s_h(\I_h z, u_h)$$ from \eqref{def:cont-problem} and \eqref{def:dis-problem} imply
	\begin{align*}
		\|u - \J_h u_h\|_{H^s(\Omega)} = a(z - \J_h \I_h z, u - \J_h u_h) + (f, \J_h\I_h z - \G_h z)_{L^2(\Omega)} + \s_h(\I_h z, u_h)&\nonumber\\
		- (\D^2_\pw (\J_h \I_h z - \mathcal{R}_h \I_h z), \D^2 \J_h u_h)_{L^2(\Omega)} - (\D^2_\pw \mathcal{R}_h \I_h z, \D^2_\pw(\J_h u_h - \mathcal{R}_h u_h))_{L^2(\Omega)}&.
	\end{align*}
	In the remaining proof, we only focus on bounding the final two terms, denoted by $\mathrm{T}_1$ and $\mathrm{T}_2$,
	on the right-hand side as the first three ones can be handled as in the proof of \Cref{thm:Hs}.
	Consider the split
	\begin{align*}
		\mathrm{T}_1 &= - (\D^2_\pw (\J_h \I_h z - \mathcal{R}_h \I_h z), \D^2_\pw \mathcal{R}_h u_h)_{L^2(\Omega)}\\
		&\qquad + (\D^2_\pw (\J_h \I_h z - \mathcal{R}_h \I_h z), \D^2_\pw(\mathcal{R}_h u_h - \J_h u_h))_{L^2(\Omega)}.
	\end{align*}
	Applying \Cref{lem:d-consistency-HHO} to the first and \Cref{lem:smoothing-operator-HHO} to the second term on the right-hand side as well as \Cref{lem:quasi-opt-rec-HHO}--\ref{lem:quasi-int-HHO} shows
	\begin{align*}
		\mathrm{T}_1 &\lesssim \|\D^2_\pw(\mathcal{R}_h u_h - \J_h u_h)\|\big(|\I_h z|_\s + \|\D^2_\pw(z - \G_h z)\|\big)\\ 
		&\lesssim \|\D^2_\pw(\mathcal{R}_h u_h - \J_h u_h)\|\|\D^2_\pw(z - \G_h z)\|.
	\end{align*}
	The triangle inequality, \Cref{lem:quasi-opt-rec-HHO}--\ref{lem:quasi-int-HHO}, the continuity of $\J_h$, and \Cref{thm:a-priori-HHO} provide
	\begin{align*}
		\|\D^2_\pw(\mathcal{R}_h u_h - \J_h u_h)\| \leq \|e_h\|_h + \|\D^2_\pw(\mathcal{R}_h \I_h u - \J_h \I_h u)\| + \|\D^2(\J_h \I_h u - \J_h u_h)\|&\\
		\lesssim \|e_h\|_h + \|\D^2_\pw(u - \G_h u)\| \lesssim \|\D^2_\pw(u - \G_h u)\| + \osc(f,\Mcal)&.
	\end{align*}
	The combination of the two previously displayed formula with \eqref{ineq:elliptic-regularity} results in
	\begin{align*}
		\mathrm{T}_1 \lesssim h_\mathrm{max}^{\min\{\delta,k-1\}}\big(\|\D^2_\pw(u - \G_h u)\| + \osc(f,\Mcal)\big).
	\end{align*}
	Since $\|\D^2_\pw(u_\Mcal - \J_h u_h)\| \lesssim \|\D^2_\pw(u - u_\Mcal)\| + \|\D^2_\pw(u - \G_h u)\| + \|e_h\|_h$ from the triangle inequality, \Cref{lem:quasi-int-HHO}, and the continuity of $\J_h$, \Cref{lem:d-consistency-HHO} and \Cref{thm:a-priori-HHO} imply
	\begin{align*}
		\mathrm{T}_2 &\lesssim \|\D^2_\pw(\mathcal{R}_h \I_h z - \J_h \I_h z)\| \big(|u_h|_\s + \|\D^2_\pw(u_\Mcal - \J_h u_h)\|\big)\\
		&\lesssim \|\D^2_\pw(\mathcal{R}_h \I_h z - \J_h \I_h z)\|\big(\|\D^2_\pw(u - \G_h u)\| + \osc(f,\Mcal)\big)
	\end{align*}
	\Cref{lem:quasi-opt-rec-HHO}--\ref{lem:quasi-int-HHO} and the elliptic regularity \eqref{ineq:elliptic-regularity} provide $\|\D^2_\pw(\mathcal{R}_h \I_h z - \J_h \I_h z)\| \lesssim h_\mathrm{max}^{\min\{\delta,k-1\}}$ and so,
	\begin{align*}
		\mathrm{T}_2 \lesssim h_\mathrm{max}^{\min\{\delta,k-1\}} \big(\|\D^2_\pw(u - \G_h u)\| + \osc(f,\Mcal)\big).
	\end{align*}
	We omit further details and refer to the proof of \Cref{thm:Hs} for the remaining arguments.
\end{proof}

\subsection{Conclusions}
We derive a~priori error estimates for nonconforming polytopal methods with Lehrenfeld-Sch\"oberl typed stabilizations for the biharmonic equation under minimal regularity assumptions. Our analysis has impact beyond a~priori error control: As a byproduct, we establish that the stabilization is an efficient contribution in a posteriori error estimators. Furthermore, lower-order error estimates allow for a corresponding error analysis of the eigenvalue problem.

\printbibliography
\end{document}